\documentclass[reqno,12pt]{amsart}
\usepackage[T1]{fontenc}
\usepackage{amssymb,amsfonts,amsbsy,graphicx,setspace,color,xfrac}
\usepackage[usenames,dvipsnames]{pstricks}
\usepackage[a4paper,margin=3truecm]{geometry}
\usepackage{calligra,ulem}

\DeclareMathAlphabet{\mathcalligra}{T1}{calligra}{m}{n}

\everymath{\displaystyle} 

 \theoremstyle{plain}
 \newtheorem{thm}{Theorem}[section]
 \newtheorem{lmm}[thm]{Lemma}
 \newtheorem{corol}[thm]{Corollary}
 \numberwithin{equation}{section} 
 \numberwithin{figure}{section} 
 \newtheorem{prop}[thm]{Proposition}
 \theoremstyle{remark}
 \newtheorem{rmk}[thm]{Remark}

 \newtheorem*{acknowledgement*}{Acknowledgement}
 \theoremstyle{definition}
 \newtheorem{defi}[thm]{Definition}

\def\I{\mathcal I}

\def\C{\mathcal C}

\def\P{\mathcal P}
\def\Pb{\mathcal P}

\newcommand{\Co}{{\bf\mathfrak{C}}}

\newcommand{\R}{\mathbb R}

\newcommand{\N}{\mathbb N}

\newcommand{\be}{\begin{equation}}
\newcommand{\ee}{\end{equation}}
\newcommand{\tu}{\tilde{u}}
\DeclareMathOperator{\spt}{spt}

\def\bx{{\bf x}}

\newcommand{\ksm}{{\bf\mathcalligra{k}} \ }
\newcommand{\kbig}{{\bf \mathcalligra{K}} \ }


\date{\today}
\author{Giuseppe Buttazzo}
\author{Thierry Champion}
\author{Luigi De Pascale}
\title[Continuity for multimarginal singular costs]{Continuity and estimates for multimarginal optimal transportation problems with singular costs}


\begin{document}
\maketitle

\begin{abstract} We consider some repulsive multimarginal optimal transportation problems which include, as a particular case, the Coulomb cost. We prove a regularity property of the minimizers (optimal transportation plan) from which we deduce existence and some basic regularity of a maximizer for the dual problem (Kantorovich potential). This is then applied to obtain some estimates of the cost and to the study of continuity properties.
\end{abstract}

\section{Introduction}

This paper deals with the following variational problem. Let $\rho\in\Pb(\R^d)$ be a probability measure, let $N>1$ be an integer and let
$$c(x_1,\dots,x_N)=\sum_{1\le i<j\le N} \phi( |x_i-x_j|), $$
with $\phi:(0,+\infty) \to \R$ satisfying the assumptions which will be described in the next section; in particular for $\phi(t)=1/t$ we have the usual Coulomb repulsive cost. Consider the set of probabilities on $\R^{Nd}$
$$\Pi(\rho)=\big\{P\in\Pb(\R^{dN})\ :\ \pi^i_\sharp P=\rho\hbox{ for all }i\big\},$$
where $\pi^i$ denotes the projection on the $i$-th copy of $\R^d$ and $\pi^i_\sharp P$ is the push-forward measure. We aim to minimize the total transportation cost
\be\label{main}
C(\rho)=\min_{P\in\Pi(\rho)}\int_{\R^{Nd}}c(x_1,\dots,x_N)\,dP(x_1,\dots,x_N).
\ee
This problem is called a multimarginal optimal transportation problem and elements of $\Pi(\rho)$ are called transportation plans for $\rho$. Some general results about multimarginal optimal transportation problems are available in \cite{carlier2003class, kellerer1984, pass2011uniqueness, pass2012local, racrus1998}. Results for particular cost functions are available, for example in \cite{gangbo1998optimal} for the quadratic cost, with some generalization in \cite{heinich2002probleme}, and in \cite{carlier2008optimal} for the determinant cost function.

Optimization problems for the cost function $C(\rho)$ in \eqref{main} intervene in the so-called {\it Density Functional Theory} (DFT), we refer to \cite{hohenberg1964, kohn1965} for the basic theory of DFT  and to \cite{gori2010density, gori2009density, seidl1999strong, seidl2007strictly, seidl1999strictly} for recent development which are of interest for us. Some new applications are emerging for example in \cite{ghoussoub2011self}. In the particular case of the Coulomb cost there are also many other open questions related to the applications. Recent results on the topic are contained in \cite{buttazzo2012, colombo2013multimarginal, colombo2013equality, cotar2013density, friesecke2013n, mendl2013kantorovich} and some of them will be better described in the subsequent sections. For more general repulsive costs we refer to the recent survey \cite{dimarino2015}. The literature quoted so far is not at all exhaustive and we refer the reader to the bibliographies of the cited papers for a more detailed picture.

Since the functions $\phi$ we consider are lower semicontinuous, the functional $\rho \mapsto C(\rho)$ above is naturally lower semicontinuous on the space of probability measures equipped with the tight convergence. In general it is not continuous as the following example shows. Take $N=2$, $\phi(t)=t^{-s}$ for some $s>0$, and
$$\rho=\frac12\delta_x+\frac12\delta_y\;,\qquad\rho_n=\Big(\frac12+\frac1n\Big)\delta_x+\Big(\frac12-\frac1n\Big)\delta_y\quad\hbox{for }n\ge1,$$
with $x\ne y$, which gives
$$C(\rho)=\frac{1}{|x-y|^s}\;,\qquad C(\rho_n)=+\infty\quad\hbox{for }n\ge1.$$
In the last section of the paper we show that the functional $C(\rho)$ is continuous or even Lipschitz on suitable subsets of $\Pb(\R^d)$ which are relevant for the applications.

Problem \eqref{main} admits the following Kantorovich dual formulation.

\begin{thm}[Proposition 2.6 in \cite{kellerer1984}] The equality
\be\label{firstdual}
\begin{split}
\min_{P\in\Pi(\rho)}\int_{\R^{Nd}}c(x_1,\dots,x_N)&\,dP(x_1,\dots,x_N) = \sup_{u\in \I_\rho}\bigg\{N\int u\,d\rho\ : \\
& u(x_1)+\dots+u(x_N)\le c(x_1,\dots,x_N) \bigg\}
\end{split}
\ee
holds, where $\I_ \rho$ denotes the set of $\rho$-integrable functions and the pointwise inequality is satisfied everywhere. 
\end{thm}

Thanks to the symmetries of the problem we also have that the right-hand side of \eqref{firstdual} coincides with
\be\label{secdual}
\sup_{u_i\in \I_\rho}\bigg\{\sum_{i=1}^N \int u_i\,d\rho\ :\ u_1(x_1)+\dots+u_N(x_N)\le c(x_1,\dots,x_N)\bigg\}.
\ee
In fact, the supremum in \eqref{secdual} is a priori larger than the one in \eqref{firstdual}; however, since for any admissible $N$-tuple $(u_1,\dots,u_N)$ the function
$$u(x)=\frac{1}{N}\sum_{i=1}^N u_i(x)$$
is admissible in \eqref{firstdual}, equality holds.

\begin{defi} A function $u$ will be called a Kantorovich potential if it is a maximizer for the right-hand-side of \eqref{firstdual}.
\end{defi}

The paper \cite{kellerer1984} contains a general approach to the duality theory for multimarginal optimal transportation problems. A different approach which make use of a weaker definition of the dual problem (two marginals case)  is introduced in \cite{beiglbock2012general} and may be applied to this situation too \cite{buttazzo2012}. Existence of Kantorovich potentials is the topic of Theorem 2.21 of \cite{kellerer1984}. That theorem requires that there exist $h_1, \dots, h_N \in L^1_\rho$ such that
$$c(x_1, \dots, x_N) \leq h_1(x_1) + \dots+ h_N(x_N),$$
and so it does not apply to the costs we consider in this paper, as for example the costs of Coulomb type. For Coulomb type costs, the existence of a Kantorovich potential is proved in \cite{depascale2015} under the additional assumption that $\rho$ is absolutely continuous with respect to the Lebesgue measure in $\R^d$. As a consequence of our main estimate, here we extend the existence result to a larger class of probability measures $\rho$. We then use the Kantorovich potentials as a tool.

We adopt the notation $\bx=(x_1,\dots,x_N)\in\R^{Nd}$ so that $x_i\in\R^d$ for each $i\in\{1,\dots,N\}$. Also, we denote the cost of a transport plan $P\in\Pi(\rho)$ by
$$\Co(P):=\int_{\R^{Nd}}c(x_1,\dots,x_N)\,dP(x_1,\dots,x_N).$$

Finally, for $\alpha>0$ we may also introduce the natural truncation of the cost
\[
c_\alpha (x_1,\dots,x_N)=\sum_{1\le i<j\le N}\min\big\{\phi(|x_i-x_j|), \phi(\alpha) \big\} = \sum_{1\le i<j\le N} \phi_\alpha(|x_i-x_j|),
\]
with the natural notation $\phi_\alpha(t):=\min\big\{\phi(t),\phi(\alpha)\big\}$ and the corresponding transportation costs $C_\alpha(\rho)$ and $\Co_\alpha(P)$.

\section{Results}
We assume that $\phi$ satisfies the following properties:
\begin{enumerate}
\item $\phi$ is continuous from $(0,+\infty)$ to $[0,+\infty)$.
\item $\lim_{t\to0^+}\phi(t)=+\infty$;
\item $\phi$ is strictly decreasing;
\end{enumerate}

\begin{rmk}
A careful attention to the present work shows that all our results also hold when (3) is replaced by the weaker hypothesis :
\begin{quote}
(3') $\phi$ is bounded at $+\infty$, that is $\sup\{ | \phi(t) |  : t \ge 1 \} <+\infty$;
\end{quote}
Under this assumption, one has to replace $\phi$ and $\phi^{-1}$ in the statements respectively by $\Phi(t) = \sup\{\phi(s) : s \ge t \}$ and
$\Phi^{-1}(t) = \inf\{s : \phi(s)=t \}$. For example in Theorem \ref{mainprop} the number $\alpha$ should be chosen lower than $\Phi^{-1}\big(\frac{N^2(N-1)}{2}\Phi(\beta)\big)$. Even if less general, we believe that the present form and stronger hypothesis (3) makes our approach and arguments more clear.
\end{rmk}

\begin{defi} For every $\rho\in\P(\R^d)$ the measure of concentration of $\rho$ at scale $r$ is defined as
$$\mu_\rho(r)=\sup_{x\in\R^d}\rho\big(B(x,r)\big).$$
\end{defi}

In particular, if $\rho\in L^1(\R^d)$ we have $\mu_\rho(r)=o(1)$ as $r\to0$, and more generally
$$\rho\in L^p(\R^d)\ \Longrightarrow\ \mu_\rho(r)=o(r^{d(p-1)/p})\quad\hbox{ as }r\to0.$$
The main role of the measures of concentration is played in the following assumption.

\medskip\noindent{\bf Assumption (A): }{\it We say that $\rho$ has small concentration with respect to $N$ if
$$\lim_{r\to0}\mu_\rho(r)<\frac{1}{N(N-1)^2}.$$}

For $\alpha>0$ we denote by 
$$D_\alpha:=\big\{\bx=(x_1,\dots,x_N)\ :\ |x_i-x_j|<\alpha\mbox{ for some }i \neq j\big\}$$
the open strip around the singular set where at least two of the $x_i$ coincide. Finally, we denote by $\R^{d(N-1)}\otimes_i A$ the Cartesian product on $N$ factors the $i$-th of which is $A$ while all the others are copies of $\R^d$.

\begin{lmm}\label{punti}
Assume that $\rho$ satisfies assumption {\bf(A)}. Let $\bx:=\bx^1\in\R^d$ and let $\beta$ be such that
$$\mu_\rho(\beta)<\frac{1}{N(N-1)^2}\;.$$
Then for every $P\in\Pi(\rho)$ there exist $\bx^2,\dots,\bx^N\in\spt P$ such that
\be\label{beta}
\beta<|x^i_j-x^k_\sigma|\qquad\mbox{ for all $j,\sigma$ whenever $k\ne i$.} \footnote{Il me semble que c'est plus pr\'ecis, et on en a besoin dans Theorem 3.4}
\ee
\end{lmm}

\begin{proof}
By definition of marginals and by the choice of $\beta$ we have
$$P\big(\R^{d(N-1)}\otimes_\alpha B(x_i^k,\beta)\big)<\frac{1}{N(N-1)^2}$$
for all indices $\alpha,i,k$. Then for any $j\in\{2,\dots,N\}$
$$P\big(\cup_{k=1}^{j-1}\cup_{i=1}^{N}\cup_{\alpha\ne i}\R^{d(N-1)}\otimes_\alpha B(x_i^k,\beta)\big)<\frac{j-1}{N-1},$$
and, since $P$ is a probability measure, this allows us to choose
$$\bx^j\in\spt P\setminus\big(\cup_{k=1}^{j-1}\cup_{i=1}^{N}\cup_{\alpha\ne i}\R^{d(N-1)}\otimes_\alpha B(x_i^k,\beta)\big).$$
It is easy to verify that the $\bx^j$ above satisfy the desired property \eqref{beta}.
\end{proof}

\subsection{Estimates for the optimal transport plans}\label{comp}

\begin{thm}\label{mainprop}
Let $\rho \in \Pb(\R^d)$ and assume that $\rho$ satisfies assumption {\bf(A)}. Let $P\in\Pi(\rho)$ be a minimizer for the transportation cost $C(\rho)$ (or $C_\alpha(\rho)$) and let $\beta$ be such that
$$\mu_\rho(\beta)<\frac{1}{N(N-1)^2}\;.$$
Then
$$\spt P\subset\R^{Nd}\setminus D_\alpha\qquad\hbox{whenever}\qquad\alpha<\phi^{-1}\Big(\frac{N^2(N-1)}{2}\phi(\beta)\Big).$$
\end{thm}

\begin{proof}
We make the proof for the case where $P$ is a minimizer for $C(\rho)$, the argument being the same for $C_\alpha(\rho)$.
Take $\alpha$ as in the statement and $\delta \in (0,\alpha)$. Note from the hypotheses on $\phi$ that $\alpha < \beta$.
Assume that $\bx^1=(x^1_1,\dots,x^1_N)\in D_{\delta} \cap\spt P$ and choose points $\bx^2,\dots,\bx^N$ in $\spt P$ as in Lemma \ref{punti}.
Let $k$ be large enough so that $\delta+2\beta/k < \alpha$.
Since all the chosen points belong to $\spt P$, we have
$$P\big(Q(\bx^i,\beta/k)\big)>0$$
where $Q(\bx^i,\beta/k) = \Pi_{j=1}^N B(\bx^i_j,\beta/k)$.
Denote by $P_i=P_{\big|Q(\bx^i,\frac{\beta}{k})}$ and choose constants $\lambda_i\in(0,1]$ such that
$$\lambda_1|P_1|=\dots=\lambda_N|P_N|:=\varepsilon,$$
where $|P_i|$ denotes the mass of the measure $P_i$. We then write
$$P=\lambda_1 P_1+\dots+\lambda_N P_N+P_R\qquad(P_R\mbox{ is the remainder}),$$
and we estimate from below the cost of $P$ as follows:
$$\Co(P)=\Co(P_R)+\sum_{i\ge1}\Co(\lambda_i P_i)\ge \Co(P_R)+\varepsilon\phi(\alpha)$$
where we used the fact that $\bx^1 \in D_\delta$ and $\delta+2\beta/k < \alpha$.
We consider now the marginals $\nu^i_1,\dots,\nu^i_N$ of $\lambda_i P_i$ and build the new local plans
$$\tilde{P}_1=\nu^1_1\times\nu^2_2\times\dots\nu^N_N,\quad
\tilde{P}_2=\nu_1^2\times\nu_2^3\times\dots\nu_N^1,\quad\dots,\quad
\tilde{P}_N=\nu_1^N\times\nu_2^1\times\dots\nu_N^{N-1}.$$
To write the estimates from above it is convenient to remark that we may also write
$$\tilde{P}_i=\nu_1^i\times\dots\nu_k^{i+k-1}\times\dots\nu_N^{i+N-1}$$
where we consider the upper index (mod $N$).
Consider now the transport plan
$$\tilde{P}:=P_R+\tilde{P}_1+\dots+\tilde{P}_N;$$
it is straightforward to check that the marginals of $\tilde{P}$ are the same as the marginals of $P$. Moreover $|\tilde{P}_i|= \lambda_i |P_i|$. So we can estimate the cost of $\tilde{P}$ from above using the distance between the coordinates of the centers of the cubes established in Lemma \ref{punti}, and we obtain
$$\Co(\tilde{P})=\Co(P_R)+\sum_{i=1}^N \Co(\tilde{P}_i)\le
\Co(P_R)+N\frac{N(N-1)}{2}\phi(\beta-2\beta/k)\varepsilon.$$
Then if
$$\alpha<\phi^{-1}\Big(\frac{N^2(N-1)}{2}\phi(\beta-\frac{2\beta}{k})\Big)$$
we have that
$$\Co(\tilde{P})<\Co(P),$$
thus contradicting the minimality of $P$. 
It follows that the strip $D_{\alpha_1}$ and $\spt P$ do not intersect if $\alpha$ satisfies the inequality above and since $k$ may be arbitrarily large and $\phi$ is continuous we obtain the conclusion for any $\alpha_1 \in  (0,\alpha)$, which concludes the proof.
\end{proof}



The Theorem above allow us to estimate the costs in term of $\beta$.

\begin{prop}\label{estimation}
Let $\rho\in\Pb(\R^d)$ and assume that $\rho$ satisfies assumption {\bf(A)}. Then if $\beta$ is such that
$$\mu_\rho(\beta)<\frac{1}{N(N-1)^2}$$
we have
$$C(\rho)\le \frac{N^3(N-1)^2}{4}\phi(\beta).$$
Moreover,
$$C(\rho)=C_\alpha(\rho)\qquad\mbox{whenever}\qquad\alpha\le\phi^{-1}\Big(\frac{N^2(N-1)}{2}\phi(\beta)\Big).$$
\end{prop}

\begin{proof} Let $P$ be an optimal transport plan for the cost $C$. According to Theorem \ref{mainprop}, if $\alpha$ is as in the statement then the support of $P$ may intersect only the boundary of $D_\alpha$ and this means that $c\le\frac{N(N-1)}{2}\phi(\alpha)$ on the support of $P$. Then
$$C(\rho)\le\int\frac{N(N-1)}{2}\phi(\alpha)\,dP=\frac{N(N-1)}{2}\phi(\alpha)$$
and, taking the largest admissible $\alpha$ we obtain
$$C(\rho)\le\frac{N^3(N-1)^2}{4}\phi(\beta),$$
which is the desired estimate. Let now $P_\alpha$ be an optimal plan for the cost $C_\alpha$, then also $\spt P_\alpha\subset\R^{Nd}\setminus D_\alpha$ so that $c=c_\alpha$ on $\spt P_\alpha$. It follows that
$$C(\rho)\le\int c\,dP_\alpha=\int c_\alpha\,dP_\alpha=C_\alpha (\rho),$$
and since the opposite inequality is always true we conclude the proof.
\end{proof}

As a consequence of Proposition \ref{estimation} above, Proposition 2.6 of \cite{kellerer1984} and Theorem 2.21 of \cite{kellerer1984} we also obtain an extension of the duality theorem of \cite{depascale2015} to a wider set of $\rho$.

\begin{thm}\label{dualita}
Let $\rho\in\Pb(\R^d)$ and assume that $\rho$ satisfies assumption {\bf(A)}. Then
\be\label{newdual}
C(\rho)=\max_{u\in \I_\rho}\bigg\{N\int u\,d\rho\ :\ 
u(x_1)+\dots+u(x_N)\le c(x_1,\dots,x_N)\bigg\}.
\ee
Moreover, whenever $\alpha\le\phi^{-1}\Big(\frac{N^2(N-1)}{2}\phi(\beta)\Big)$, any Kantorovich potential $u_\alpha$ for $C_\alpha$ is also a Kantorovich potential for $C$.
\end{thm}

\begin{proof}
By monotonicity of the integral the left-hand side of \eqref{newdual} is always larger than the right-hand side. Proposition 2.6 and Theorem 2.21 of \cite{kellerer1984} may be applied to the cost $c_\alpha$ to obtain
$$C_\alpha (\rho)=\max_{u\in \I_\rho}\bigg\{N\int u\,d\rho\ :\ 
u(x_1)+\dots+u(x_N)\le c_\alpha(x_1,\dots,x_N)\bigg\}.$$
Since $\rho$ satisfies assumption {\bf(A)}, by Proposition \ref{estimation} for $\alpha$ sufficiently small we have that there exists $u_\alpha \in \I_\rho$ such that
$$u_\alpha (x_1)+\dots+u_\alpha(x_N)\le c_\alpha(x_1,\dots,x_N)\le c(x_1,\dots,x_N),$$
and
$$C(\rho)=C_\alpha(\rho)=N\int u_\alpha\,d\rho\;,$$
as required.
\end{proof}

\begin{rmk}\label{complementary}
Note that if $u$ is a Kantorovich potential for $C$ and $P$ is optimal for $C$ then
$u(x_1)+\ldots+u(x_N)=c(x_1,\ldots,x_N)$ holds $P$-almost everywhere.
\end{rmk}

\section{Applications}

\subsection{Estimates for the cost} 
Since the parameter $\beta$ in the previous section is naturally related to the summability of $\rho$, we can obtain some estimate of the cost $C(\rho)$ in term of the available
norms of $\rho$.

\begin{prop}\label{lp}
Let $\rho\in\Pb(\R^d)\cap L^p(\R^d)$ for some $p>1$. Then, if $P\in\Pi(\rho)$ is optimal for the transportation cost $C(\rho)$, we have
$$P(D_\alpha)=0\quad\mbox{whenever}\quad\alpha < \phi^{-1}\bigg(\frac{N^2(N-1)}{2} \phi\bigg(\bigg(\frac{1}{\omega_d\big(N(N-1)^2\big)^{p'}\|\rho\|_p^{p'}}\bigg)^{1/d}\bigg)\bigg),$$
where $\omega_d$ denotes the Lebesgue measure of the ball of radius 1 in $\R^d$ and $p'$ the conjugate exponent of $p$. It follows that
\be\label{eq21}
C(\rho)\le \frac{N^3(N-1)^2}{4} \phi\bigg(\bigg(\frac{1}{\omega_d\big(N(N-1)^2\big)^{p'}\|\rho\|_p^{p'}}\bigg)^{1/d}\bigg) .
\ee
\end{prop}

\begin{proof}
Let
$$\beta\le\bigg(\frac{1}{\omega_d\big(N(N-1)^2\big)^{p'}\|\rho\|_p^{p'}}\bigg)^{1/d}.$$
By H\"older inequality we have
$$\int_{B(x,\beta)}\rho(y)\,dy\le\|\rho\|_p(\omega_d\beta^d)^{1/p'}\le\frac{1}{N(N-1)^2}\;,$$
so that
$$\mu_\rho(\beta)\le\frac{1}{N(N-1)^2}\;.$$
The desired inequality \eqref{eq21} now follows by Theorem \ref{estimation}.
\end{proof}

\begin{rmk} The Coulomb type costs $\phi(t)=t^{-s}$ for $s>0$ play a relevant role in several applications.
\begin{enumerate}
\item For $\phi(t)=t^{-s}$ estimate \eqref{eq21} above takes the form 
$$C(\rho)\le\frac{N^3(N-1)^2}{4}\bigg(\omega_d\big(N(N-1)^2\big)^{p'}\|\rho\|_p^{p'}\bigg)^{s/d}.$$

\item In dimension $d=3$ and for $s=1$ the set
$${\mathcal H}:=\Big\{\rho\in L^1(\R^3)\ :\ \rho\ge0,\ \sqrt{\rho}\in H^1(\R^3),\ \int\rho\,dx=1\Big\}$$ 
plays an important role in the Density Functionals Theory. In fact, Lieb in \cite{lieb1983} proved that $\rho\in {\mathcal H}$ if and only if there exists a wave function $\psi\in H^1(\R^{3N})$ such that
$$\pi^i_\sharp|\psi|^2dx=\rho,\qquad\mbox{for }i=1,\dots,N.$$
Taking $s=1$, $d=3$, $p=3$ in Proposition \ref{lp} gives \\
$$C(\rho)\le C N^{7/2} (N-1)^3\|\rho\|_3^{1/2}=C N^{7/2} (N-1)^3\|\sqrt{\rho}\|_6\le C N^{7/2} (N-1)^3\|\sqrt{\rho}\|_{H^1}.$$ \\
\end{enumerate}
\end{rmk}

\subsection{Estimates for Kantorovich potentials}
In general, a Kantorovich potential $u$ is a $\rho$-integrable function which can be more or less freely modified outside a relevant set. In this section we show the existence of Kantorovich potentials which are more regular. 

\begin{lmm}\label{infconv} Let $u$ be a Kantorovich potential; then there exists a Kantorovich potential $\tu$ which satisfies
$$u\le\tu,$$
and
\be\label{uinfconv}
\tu(x)=\inf\bigg\{c(x,y_2,\dots,y_N)-\sum_{j\ge2}\tu(y_j)\ :\ y_j\in\R^d\bigg\}\qquad\forall x\in\R^d.
\ee
\end{lmm}

\begin{proof}
We first define 
$$\overline{u}(x):=\inf\bigg\{c(x,y_2,\dots,y_N)-\sum_{j\ge2}u(y_j)\ :\ y_j\in\R^d\bigg\};$$
then we consider
$$\hat{u}(x)=\frac{\overline{u}(x)+(N-1)u(x)}{N}.$$
Since $u(x)\le\overline{u}(x)$ we have also $u(x)\le\hat{u}(x)$ ; moreover it is straightforward to check that
$$\hat{u}(x_1)+\dots+\hat{u}(x_N)\le c(x_1,\dots,x_N)\qquad\forall x_i\in\R^d.$$
Notice that if $u$ does not satisfy \eqref{uinfconv} at some $x$ then $u(x) < \hat{u}(x)$. We then consider
$$\tu(x)=\sup\big\{v(x)\ :\ u\le v\ \rho-\mbox{a.e. and }v(x_1)+\dots+v(x_n)\le c(x_1,\dots,x_N)\big\}$$
which satisfies all the required properties (since $\tu = \hat{\tu}$).
\end{proof}

Taking some constant $\alpha_1,\dots,\alpha_N$ such that $\sum\alpha_i=0$ we may define $u_i(x)=\tu(x)+\alpha_i$ and we obtain an $N$-tuple of functions which is optimal for problem \eqref{secdual} and satisfies
$$u_i(x)=\inf\bigg\{c(y_1,\dots,y_{i-1},x,y_{i-1},\dots,y_N)-\sum_{j\ne i}u_j(y_j)\ :\ y_j\in\R^d\bigg\}.$$
The choice of the constants $\alpha_i$ can be made so that the functions $u_i$ take specific and admissible values at some points. A final, elementary, remark is that $\tu$ is the arithmetic mean of the $u_i$s.
 
\begin{thm}\label{stimeuniformi}
Let $\rho\in\Pb(\R^d)$. Assume that $\rho$ satisfies assumption {\bf(A)}, and let $\beta$ be such that $\mu_\rho(\beta)\le\frac{1}{N(N-1)^2}$. Let $u$ be a Kantorovich potential which satisfies
$$u(x)=\inf\bigg\{c(x,\dots,y_2,\dots,y_N)-\sum_{j\ge2}u(y_j)\ :\ y_j\in\R^d\bigg\}.$$
Then for any choice of $\alpha$ as in Theorem \ref{mainprop} it holds
\be\label{normasup}
\sup_{\R^d}|u|\le N(N-1)^2\phi\left(\frac{\alpha}{2}\right).
\ee
\end{thm}

\begin{proof}
Let $P$ be an optimal transport plan for $C$, let $\alpha$ be as in Theorem \ref{mainprop} and take $\overline{\bx} \in \spt P$, then
$|\overline{\bx}_i-\overline{\bx}_j|\geq \alpha$ for $i\ne j$. From Remark \ref{complementary} we can assume that
$u(\overline{\bx}_1)+\ldots+u(\overline{\bx}_N) = c(\overline{\bx}_1,\dots,\overline{\bx}_N)$. From the above discussion we may consider a Kantorovich $N$-tuple $(u_1,\dots,u_N)$ obtained from $u$ which is optimal for \eqref{secdual} and satisfies 
$$u_i(x)=\inf\bigg\{c(y_1,\dots,y_{i-1},x,y_{i-1},\dots,y_N)-\sum_{j\ne i}u_j(y_j)\ :\ y_j\in\R^d\bigg\}\qquad\hbox{for all }x$$
and $u_i(\overline{\bx}_i)=\frac{1}{N}c(\overline{\bx}_1,\dots,\overline{\bx}_N) \geq 0$ for all $i$.
If $x\notin\cup_{i=2}^N B(\overline{\bx}_i,\frac{\alpha}{2})$ it holds:
$$u_1(x)\le c(x,\overline{\bx}_2,\dots,\overline{\bx}_N)-\sum_{j\ge2}u_j(\overline{\bx}_j) \le c(x,\overline{\bx}_2,\dots,\overline{\bx}_N)\le\frac{N(N-1)}{2} \phi(\frac{\alpha}{2}).$$
Taking $\bx^1 = \overline{\bx}$, we apply Lemma \ref{punti} and obtain a point $\bx^2 \in \spt P \setminus D_\alpha$ such that
$|\bx^2_j - \overline{\bx}_\sigma| \geq \beta \geq \alpha$ for all $j$ and $\sigma$, and from Remark \ref{complementary} we may assume that
\[
- \sum_{j =2}^N u_j(\bx^2_j)= u_1(\bx^2_1) - c(\bx^2_1,\dots,\bx^2_N) \le u_1(\bx^2_i) \le \frac{N(N-1)}{2}\phi(\frac{\alpha}{2})
\]
where we used $\bx^2_1 \notin\cup_{j \ge 2} B(\overline{\bx}_j,\alpha)$. Fix $i \geq 2$, it follows that if $x\in B(\overline{\bx}_i,\frac{\alpha}{2})$ then
$|x - \bx^2_j| \ge \frac{\alpha}{2}$ for all $j \ge 2$ so that
\[
u_1(x)  \le  c(x,\bx^2_2,\dots,\bx^2_N)-\sum_{j =2}^N u_j(\bx^2_j) \le N(N-1) \phi(\frac{\alpha}{2}) 
\]
This concludes the estimate from above for $u_1$ on $\R^d$, and analogously for all $u_i$.
The formula above now allows us to find an estimate from below which, again, we write for $u_1$ as
$$u_1(x)=\inf\bigg\{c(x,y_2,\dots,y_N)-\sum_{j\ge2}u_j(y_j)\ :\ y_j\in\R^d\bigg\} \ge- N(N-1)^2 \phi(\frac{\alpha}{2}).$$
Then for all $i$ one has 
$$\|u_i\|_\infty\le N(N-1)^2\phi(\frac{\alpha}{2})\;,$$
and analogously for $u=\frac{1}{N}\sum_i u_i$ the same estimate holds.
\end{proof}

\begin{rmk} Theorem \ref{stimeuniformi} above applies to all costs considered in this paper including $c_\alpha$ obtained replacing the function $\phi$ by its truncation $\phi_\alpha$.
\end{rmk}

The next theorem shows that under the usual assumptions on $\rho$ and some additional assumptions on $\phi$ there exists a Kantorovich potential which is Lipschitz and semiconcave with Lipschitz and semiconcavity constants depending on the concentration of $\rho$. In the next statement we denote by $Sc(u)$ the semiconcavity constant, that is the lowest nonnegative constant $K$ such that $u-K|\cdot|^2$ is concave.

\begin{thm}\label{lipsemiconc}
Let $\rho\in\Pb(\R^d)$. Assume that $\rho$ satisfies assumption {\bf(A)}, and let $\beta$ be such that $\mu_\rho(\beta)\le\frac{1}{N(N-1)^2}$.
\begin{itemize}
\item If $\phi$ is of class $\C^1$ and for all $t>0$ there exists a constant $\ksm(t)$ such that
\be\label{uLip}
|\phi'(s)|<\ksm(t)\qquad\hbox{for all }s>t
\ee
then there exists a Kantorovich potential $u$ for problem \eqref{firstdual} such that
$$Lip(u)\le\ksm\left(\phi^{-1}\Big(\frac{N^2(N-1)}{2}\phi(\beta)\Big)\right).$$
\item If $\phi$ is of class $\C^2$ and for all $t>0$ there exists a constant $\kbig(t)$ such that
$$\phi''(s)-\frac{\phi'(s)}{s}<\kbig(t)\qquad\hbox{for all }s>t$$
then there exists a Kantorovich potential $u$ for problem \eqref{firstdual} such that
$$Sc(u)\le\kbig\left(\phi^{-1}\Big(\frac{N^2(N-1)}{2}\phi(\beta)\Big)\right).$$
\end{itemize}
\end{thm}

\begin{proof}
According to Proposition \ref{estimation} and Theorem \ref{dualita} for $\alpha\le\phi^{-1}\Big(\frac{N^2(N-1)}{2}\phi(\beta)\Big)$ any Kantorovich potential $u_\alpha$ for the cost $\Co_\alpha$ is also a potential for $\Co$. According to Lemma \ref{infconv} above we choose a potential $u_\alpha$ satisfying
$$\forall x, \quad u_\alpha(x)=\inf\bigg\{c_\alpha (x,\dots,y_2,\dots,y_N)-\sum_{j\ge2}u(y_j)\ :\ y_j \in\R^d\bigg\}.$$
Since the infimum of uniformly Lipschitz (resp. uniformly semiconcave) functions is still Lipschitz (resp. semiconcave) with the same constant, it is enough to show that the functions
$$x\mapsto c_\alpha(x,\dots,y_2,\dots,y_N)+C$$
are uniformly Lipschitz and semiconcave. To check that it is enough to compute the gradient and the Hessian matrix of these functions and use the respective properties
of the pointwise cost $\phi$.
\end{proof}

\begin{rmk} The above Theorem \ref{lipsemiconc} applies to the Coulomb cost $\phi(t)=1/t$ and more generally to the costs  $\phi(t)=t^{-s}$ for $s>0$.
\end{rmk}

\subsection{Continuity properties of the cost}

In this subsection we study some conditions that imply the continuity of the transportation cost $C(\rho)$ with respect to the tight convergenge on the marginal variable $\rho$. 

\begin{lmm}\label{equidist} Let $\{\rho_n\}\subset\P(\R^3)$ be such that $\rho_n\stackrel{*}{\rightharpoonup}\rho$ and assume that $\rho$ satisfies assumption {\bf(A)}. Let $\beta$ be such that
$$\mu_\rho(\beta)<\frac{1}{N(N-1)^2}.$$
Then for all $\delta\in(0,1)$ there exists $k\in\N$ such that for all $n>k$
$$\mu_{\rho_n}(\delta\beta)<\frac{1}{N(N-1)^2}.$$
\end{lmm}

\begin{proof} 
We argue by contradiction assuming that there exists a sequence $\{x_n\}$ such that
$$\frac{1}{N(N-1)^2} \leq \rho_n\big(B(x_n,\delta\beta)\big).$$
Since the sequence $\{\rho_n\}$ is uniformly tight there exists $K$ such that $|x_n|<K$. Up to subsequences we may assume that $x_n\to\tilde{x}$ for a suitable $\tilde{x}$.
Let $\delta' \in (\delta,1)$.
Then, for $n$ large enough, $\overline{B(x_n,\delta\beta)}\subset B(\tilde{x},\delta' \beta)$ and since
$$\rho_n\big(B(x_n,\delta\beta)\big)\le\rho_n\big(\overline{B(x_n,\delta\beta)}\big)\le\rho_n\big(\overline{B(x,\delta' \beta)}\big),$$
and 
$$\limsup\rho_n\big(\overline{B(x,\delta' \beta)}\big)\le\rho\big(\overline{B(x,\delta' \beta)}\big) \le \rho\big(B(x,\beta)\big)$$
we obtain
$$\limsup\rho_n\big(B(x,\delta\beta)\big) < \frac{1}{N(N-1)^2},$$
which is a contradiction.
\end{proof}

\begin{thm}\label{wcont}
Let $\{\rho_n\}\subset\P(\R^3)$ be such that $\rho_n\stackrel{*}{\rightharpoonup}\rho$ with $\rho$ satisfying assumption {\bf(A)}.
Assume that the cost function $\phi$ satiisfies assumption \eqref{uLip}.
Then
$$C(\rho_n)\to C(\rho) \quad \mbox{as $n \to +infty$.}$$
\end{thm}

\begin{proof}
We first note that the theorem above holds for the costs $c_\alpha$ since they are continuous and bounded in $\R^d$, and from the fact that whenever
$P \in \Pi(\rho)$ there exists $P_n \in \Pi(\rho_n)$ such that  $P_n\stackrel{*}{\rightharpoonup} P$, so that $C_\alpha$ is continuous with respect to weak convergence.

Thanks to Lemma \ref{equidist} and Theorem \ref{mainprop} we infer that there exists $k>0$ and $\alpha >0$ such that
the optimal transport plans for $C(\rho)$ and $C(\rho_n)$ all all supported in $\R^{Nd} \setminus D_\alpha$
for $n \geq k$. But then the functionals $C$ and $C_\alpha$ coincide on $\{\rho\}\cup \{\rho_n\}_{n \geq k}$ and the thesis follows form the continuity of $C_\alpha$.
\end{proof}

\begin{rmk} Under the hypothesis \eqref{uLip} on $\phi$ we may propose the following alternative proof for Theorem \ref{wcont} above.
Since the pointwise cost $c$ is lower semicontinuous, by the dual formulation \eqref{firstdual} the functional $C$ is lower semicontinuous too. Then we only need to prove the inequality
$\limsup_{n\to\infty}C(\rho_n)\leq C(\rho)$.
By Theorems \ref{stimeuniformi} and \ref{lipsemiconc} and Lemma \ref{equidist} above, there exists a constant $K$ and an integer $\nu$ such that for $n\ge\nu$ we can choose a Kantorovich potential $u_n$ for $\rho_n$ and the cost $C$ which is $K$-Lipschitz and bounded by $K$. Up to subsequences we may assume that $u_n\to u$ uniformly on compact sets, so $u$ is $K$-Lipschitz and satisfies
$$u(x_1)+\dots+u(x_N)\le c(x_1,\dots,x_N).$$
It follows that
$$N\int u\,d\rho\le C(\rho)\le\lim_{n\to\infty}C(\rho_n)=\lim_{n\to\infty}N\int u_n\,d\rho_n=N\int u\,d\rho$$
as required.
\end{rmk}

\begin{thm}
Let $\rho_1,\rho_2\in\P(\R^3)$ be such that
$$\mu_{\rho_i}(\beta)<\frac{1}{N(N-1)^2}\qquad i=1,2,$$
for a suitable $\beta>0$. Then for every $\alpha$ as in Theorem \ref{mainprop} we have
$$|C(\rho_1)-C(\rho_2)|\le N^2(N-1)^2 \phi(\frac{\alpha}{2})\|\rho_1-\rho_2\|_{L^1}.$$
\end{thm}

\begin{proof} Without loss of generality we may assume that $C(\rho_2) \le C(\rho_1)$. Let $u_1$ and $u_2$ be Kantorovich potentials which satisfy the estimate of Theorem \ref{stimeuniformi} respectively for $\rho_1$ and $\rho_2$ . We have
$$ C (\rho_1) -C(\rho_2) = N \int u_1 d \rho_1 -N \int u_2 d \rho_2 .$$
By the optimality of $u_1$ and $u_2$
$$N\int u_2\,d(\rho_1-\rho_2)\le N\int u_1\,d\rho_1-N\int u_2\,d\rho_2\le N\int u_1\,d(\rho_1-\rho_2)$$
and the conclusion now follows by estimate \eqref{normasup}.
\end{proof}

\begin{corol}  The functional $C(\rho)$ is Lipschitz continuous on any bounded subset of $L^p(\R^d)$ for $p>1$ and in particular on any bounded subset of the space ${\mathcal H}$. 
\end{corol}

\section*{Acknowledgement}

This paper has been written during some visits of the authors at the Department of Mathematics of University of Pisa and at the Laboratoire IMATH of University of Toulon. The authors gratefully acknowledge the warm hospitality of both institutions.

The research of the first and third authors is part of the project 2010A2TFX2 {\it Calcolo delle Variazioni} funded by the Italian Ministry of Research and is partially financed by the {\it``Fondi di ricerca di ateneo''} of the University of Pisa.

%

\bibliographystyle{plain}

\begin{thebibliography}{10}

\bibitem{beiglbock2012general}
M. Beiglb\"ock, C. L\'eonard, W. Schachermayer.
\newblock A general duality theorem for the monge--kantorovich transport problem.
\newblock {\em Studia Math.}, 209:2, 2012.

\bibitem{buttazzo2012}
G.~Buttazzo, L.~De~Pascale, and P.~Gori-Giorgi.
\newblock Optimal-transport formulation of electronic density-functional theory.
\newblock {\em Physical Review A}, 85(6):062502, 2012.

\bibitem{carlier2003class}
G. Carlier.
\newblock On a class of multidimensional optimal transportation problems.
\newblock {\em J. Convex Anal.}, 10(2):517--530, 2003.

\bibitem{carlier2008optimal}
G. Carlier, B. Nazaret.
\newblock Optimal transportation for the determinant.
\newblock {\em ESAIM Control Optim. Calc. Var.}, 14(04):678--698, 2008.

\bibitem{colombo2013multimarginal}
M. Colombo, L. De~Pascale, S. Di~Marino.
\newblock Multimarginal optimal transport maps for 1-dimensional repulsive costs.
\newblock {\em Canad. J. Math}, 2013.

\bibitem{colombo2013equality}
M. Colombo, S. Di~Marino.
\newblock Equality between Monge and Kantorovich multimarginal problems with Coulomb cost.
\newblock {\em Ann. Mat. Pura Appl.}, pages 1--14, 2013.

\bibitem{cotar2013density}
C. Cotar, G. Friesecke, C. Kl\"uppelberg.
\newblock Density functional theory and optimal transportation with Coulomb cost.
\newblock {\em Comm. Pure Appl. Math.}, 66(4):548--599, 2013.

\bibitem{depascale2015}
L. De Pascale,
\newblock Optimal transport with {C}oulomb cost. {A}pproximation and duality.
\newblock{ \em ESAIM Math. Model. Numer. Anal.}, 49(6):1643--1657, 2015.

\bibitem{dimarino2015}
S. Di Marino, A. Gerolin, L. Nenna.
\newblock Optimal Transportation Theory with Repulsive Costs.
\newblock{ \em http://arxiv.org/abs/1506.04565}.

\bibitem{friesecke2013n}
G. Friesecke, C.B. Mendl, B. Pass, C. Cotar, C. Kl\"uppelberg.
\newblock N-density representability and the optimal transport limit of the Hohenberg-Kohn functional.
\newblock {\em The Journal of chemical physics}, 139(16):164109, 2013.

\bibitem{gangbo1998optimal}
W. Gangbo, A. Swiech.
\newblock Optimal maps for the multidimensional Monge-Kantorovich problem.
\newblock {\em Comm. Pure Appl. Math.}, 51(1):23--45, 1998.

\bibitem{ghoussoub2011self}
N. Ghoussoub, A. Moameni.
\newblock A self-dual polar factorization for vector fields.
\newblock {\em Comm. Pure Appl. Math.}, 66 (6) (2013), 905--933.

\bibitem{gori2010density}
P. Gori-Giorgi, M. Seidl.
\newblock Density functional theory for strongly-interacting electrons: perspectives for physics and chemistry.
\newblock {\em Physical Chemistry Chemical Physics}, 12(43):14405--14419, 2010.

\bibitem{gori2009density}
P. Gori-Giorgi, M. Seidl, G. Vignale.
\newblock Density-functional theory for strongly interacting electrons.
\newblock {\em Physical review letters}, 103(16):166402, 2009.

\bibitem{heinich2002probleme}
H. Heinich.
\newblock Probl\`eme de Monge pour $n$ probabilit\'es.
\newblock {\em C. R. Math. Acad. Sci. Paris}, 334(9):793--795, 2002.

\bibitem{hohenberg1964}
P. Hohenberg, W. Kohn.
\newblock Inhomogeneous electron gas.
\newblock {\em Physical review}, 136(3B):B864, 1964.

\bibitem{kellerer1984}
H.G. Kellerer.
\newblock Duality theorems for marginal problems.
\newblock {\em Probab. Theory Related Fields}, 67(4):399--432, 1984.

\bibitem{kohn1965}
W. Kohn, L.J. Sham.
\newblock Self-consistent equations including exchange and correlation effects.
\newblock {\em Physical Review}, 140(4A):A1133, 1965.

\bibitem{lieb1983}
E.H. Lieb.
\newblock Density functionals for coulomb systems.
\newblock {\em International journal of quantum chemistry}, 24(3):243--277, 1983.

\bibitem{mendl2013kantorovich}
C.B. Mendl, L. Lin.
\newblock Kantorovich dual solution for strictly correlated electrons in atoms and molecules.
\newblock {\em Physical Review B}, 87(12):125106, 2013.

\bibitem{pass2011uniqueness}
B. Pass.
\newblock Uniqueness and monge solutions in the multimarginal optimal transportation problem.
\newblock {\em SIAM J. Math. Anal.}, 43(6):2758--2775, 2011.

\bibitem{pass2012local}
B. Pass.
\newblock On the local structure of optimal measures in the multi-marginal optimal transportation problem.
\newblock {\em Calc. Var. Partial Differential Equations}, 43(3-4):529--536, 2012.

\bibitem{racrus1998}
S.T. Rachev, L. R\"uschendorf.
\newblock {\em Mass transportation problems. Vol. I: Theory}.
\newblock Probability and its Applications (New York). Springer-Verlag, New York, 1998.

\bibitem{seidl1999strong}
M. Seidl.
\newblock Strong-interaction limit of density-functional theory.
\newblock {\em Physical Review A}, 60(6):4387, 1999.

\bibitem{seidl2007strictly}
M. Seidl, P. Gori-Giorgi, A. Savin.
\newblock Strictly correlated electrons in density-functional theory: A general formulation with applications to spherical densities.
\newblock {\em Physical Review A}, 75(4):042511, 2007.

\bibitem{seidl1999strictly}
M. Seidl, J.P. Perdew, M. Levy.
\newblock Strictly correlated electrons in density-functional theory.
\newblock {\em Physical Review A}, 59(1):51, 1999.

\end{thebibliography}

\bigskip
{\small\noindent
Giuseppe Buttazzo:
Dipartimento di Matematica,
Universit\`a di Pisa\\
Largo B. Pontecorvo 5,
56127 Pisa - ITALY\\
{\tt buttazzo@dm.unipi.it}\\
{\tt http://www.dm.unipi.it/pages/buttazzo/}

\bigskip\noindent
Thierry Champion:
Laboratoire IMATH,
Universit\'e de Toulon\\
CS 60584,
83041 Toulon cedex 9 - FRANCE\\
{\tt champion@univ-tln.fr}\\
{\tt http://champion.univ-tln.fr}

\bigskip\noindent
Luigi De Pascale:
Dipartimento di Matematica,
Universit\`a di Pisa\\
Largo B. Pontecorvo 5,
56127 Pisa - ITALY\\
{\tt depascal@dm.unipi.it}\\
{\tt http://www.dm.unipi.it/\textasciitilde depascal/}

\end{document}